\documentclass[pdflatex,sn-chicago,orcid=false,natbib]{sn-jnl}

\usepackage[authoryear]{natbib}

\usepackage{graphicx}
\usepackage{multirow}
\usepackage{amsmath,amssymb,amsfonts}
\usepackage{amsthm}
\numberwithin{equation}{section}
\usepackage{mathrsfs}
\usepackage[title]{appendix}
\usepackage{xcolor}
\usepackage{textcomp}
\usepackage{manyfoot}
\usepackage{booktabs}
\usepackage{algorithm}
\usepackage{algorithmicx}
\usepackage{algpseudocode}
\usepackage{listings}
\usepackage{float}
\usepackage{url}

\theoremstyle{thmstyleone}
\newtheorem{theorem}{Theorem}

\newtheorem{lemma}[theorem]{Lemma}
\newtheorem{corollary}[theorem]{Corollary}

\theoremstyle{thmstyletwo}

\newtheorem{remark}{Remark}

\theoremstyle{thmstylethree}
\newtheorem{definition}{Definition}

\newcommand{\diam}{\operatorname{diam}}

\usepackage[section]{placeins}  
\usepackage{array,booktabs}     
\newcolumntype{L}[1]{>{\raggedright\arraybackslash}p{#1}}

\begin{document}

\title[Multifractality in the Tree of Life]{Multifractality in the Tree of Life: A~Branching-Process RIFS Proof}

\author*[{1}]{\fnm{Kevin} \sur{Hudnall}}\email{kahudnall@ucdavis.edu}

\affil*[1]{\orgdiv{Biological Systems Engineering Graduate Group}, \orgname{University of California, Davis}, \orgaddress{\street{One Shields Avenue}, \city{Davis}, \state{CA}, \postcode{95616}, \country{USA}}}

\abstract{

We study a branching-process random iterated function system (RIFS) defined by a recursive replacement of leaves by finite subtrees at strictly smaller contraction scales. This construction yields a tree-valued, infinite-depth random geometry that unifies classical branching processes and random iterated function systems while remaining distinct from both. We prove rigorously that the resulting branching-process RIFS is multifractal under explicit and mild assumptions. Two variants are analyzed: a non-anchored case with a nontrivial compact attractor, and a biologically motivated anchored case in which the invariant geometric set collapses to a point, while tangent measures obey the same multifractal law. The construction formalizes the foundational principles of nestedness, duality, and randomness in the living tree of life~\citep{HudnallDSouza2025}, yielding a minimal-condition theorem that explains the ubiquity of multifractal signatures in biological data.

}

\keywords{Random iterated function systems, Hausdorff dimension, Multifractal formalism, Gibbs measures, Scaling laws in biology}

\maketitle

\section{Introduction}\label{sec:intro}

It has often been argued that the tree of life exhibits fractal structure, with evidence drawn from phylogenies, ecological scaling laws, and morphological complexity \citep{Burlando1993,Chaline2010,Dubois1992,Green1991,Nottale2002}. Independently, three principles are widely regarded as foundational to the organization of living systems: nestedness, duality, and randomness \citep{Kimura1983,Lewontin1970,SzathmarySmith1995}. These principles are reviewed in Section~\ref{subsec:three-principles}. Here we show that, when taken together, they give rise to a natural mathematical construction whose invariant measures are necessarily \emph{multifractal}.

Random fractal models have long provided mathematical structure to aspects of this picture. Branching processes \citep{AthreyaNey1972,KimmelAxelrod2015} capture stochastic reproduction and genealogical dependence, and in this sense already reflect randomness and a form of nestedness. However, they do not encode the recursive whole--part duality characteristic of biological organization. Random iterated function systems (RIFS) \citep{Falconer2014,PrusinkiewiczLindenmayer1996} represent recursive geometric nesting, but lack genealogical dependence and do not enforce whole--part duality. Each framework captures only part of the biological structure.

The branching-process random iterated function system (RIFS) closes this gap. In this construction, each leaf of a branching process generates an entire finite subtree at a strictly smaller contraction scale, producing a recursively embedded hierarchy in which nestedness, duality, and randomness are simultaneously present. The object was recently introduced as a mathematical construction in its own right \citep{HudnallDSouza2025}, generalizing both branching-process trees and classical RIFS.

Unlike standard phylogenetic models such as coalescent theory \citep{Kingman1982} or birth--death formulations \citep{Nee1992}, which produce static genealogies, the branching-process RIFS is inherently dynamic: every leaf both terminates a lineage and initiates a new one, generating a recursively self-similar structure that evolves across scales.

In this paper we prove that such systems necessarily exhibit multifractality, establishing a new class of random fractals motivated by biological evolution. The remainder of the paper is organized as follows. Section~\ref{sec:model} formalizes the construction. Section~\ref{sec:results} states and proves the main results. Section~\ref{sec:discussion} situates these results within the broader theory of random fractals and biological modeling, and Section~\ref{sec:conclusion} reflects on their mathematical and biological significance. Because nestedness, duality, and randomness are already well established in biological theory, the analysis shows that multifractality is not an additional modeling assumption, but a mathematical consequence of these foundational principles.

\subsection{Main result}\label{subsec:main-result}

The central result of this paper is that the branching-process RIFS is multifractal under the assumptions stated in Section~\ref{subsec:assumptions}. We analyze two variants of the construction. In the first, subtrees are embedded freely within recursively shrinking subintervals of the unit interval, yielding a nontrivial compact attractor whose invariant measure has a strictly convex $L^q$-spectrum and hence a nontrivial multifractal spectrum \citep{Mandelbrot1974}. 

In the second variant, the nested subtrees are anchored to the boundary point $0$, so that in the infinite limit the invariant geometric set collapses to the singleton $\{0\}$. This anchored formulation is biologically motivated: living systems are finite and never infinitely resolved, with collapse corresponding to the inevitability of death. Even in this case, multifractality persists through finite-depth approximations and tangent measures. Thus the property holds in both formulations—directly in the invariant measure in the non-anchored case, and through rescaled local structure in the anchored case.

Together, these results establish a new example of multifractality in random fractals that is directly relevant to modeling biological evolution through the combination of branching processes and iterated contractions. The present work therefore provides a rigorous proof of multifractality in a genealogically dynamic setting, complementing the simulations and heuristic analysis of \citep{HudnallDSouza2025}.

\section{Model definition}\label{sec:model}

This section defines the branching-process random iterated function system studied throughout the paper. We begin with a high-level description of the construction, then give a precise formal definition, introduce two natural variants, and situate the model relative to existing branching and random fractal frameworks.

\subsection{Structural overview}

The object studied here is not a Galton--Watson process evolving generation by generation, nor a classical random iterated function system. Instead, the recursion acts on entire trees. At each stage of the construction, every leaf of the current tree is replaced by a randomly generated finite Galton--Watson subtree that is embedded at a strictly smaller geometric scale. Each such replacement produces a new tree-valued object, and the process then repeats on the newly created leaves. While each inserted subtree has almost surely finite height, the recursive replacement of leaves continues indefinitely, yielding an infinite-depth structure. Geometry and genealogy are therefore coupled dynamically: the branching process does not merely label an existing geometric cascade, but actively generates new geometric structure at each scale. This recursive replacement mechanism distinguishes the branching-process random iterated function system from both standard branching processes (which lack geometry) and classical random iterated function systems (which lack genealogical dependence).

\subsection{Branching-process random iterated function system}

We now give a precise definition of the branching-process RIFS, which formalizes the recursive tree-valued construction described above.

\begin{definition}[Branching-process random iterated function system]
A \emph{branching-process random iterated function system} (branching-process RIFS) is a recursive random construction defined as follows. One begins with a rooted tree whose root is associated with a compact geometric support (here the unit interval). At each iteration, every leaf of the current tree is independently replaced by a randomly generated \emph{finite} Galton--Watson subtree. This subtree is embedded inside the geometric support of its parent leaf by a family of similarity contractions whose total scale is strictly smaller than that of the parent. The replacement operation produces a new tree with a finer geometric embedding, and the procedure is then repeated on the newly created leaves.

The recursion therefore acts on trees rather than on individual vertices: each leaf simultaneously terminates a lineage at its parent scale and initiates an entire subtree at a smaller scale. Repeated replacement generates an infinite-depth, hierarchically embedded structure in which genealogical branching and geometric contraction are intrinsically coupled.
\end{definition}

All probabilistic assumptions on offspring distributions and contraction factors are summarized in Section~\ref{subsec:assumptions}.

We recall the construction introduced in \cite{HudnallDSouza2025}. Let $X = [0,1]$. At depth $k = 0$, the root of the process is identified with the whole interval
$I(\mathrm{root}) = X$, and an initial Galton--Watson tree is generated with the root scale fixed at $s(\mathrm{root}) = 1$ (see Appendix~\ref{App:C} for a compact notation table).

At each subsequent step $k \geq 1$, every leaf of the current tree at depth $k-1$ is replaced by a finite Galton--Watson subtree that is embedded inside a subinterval of strictly smaller diameter than that of its parent leaf. For each leaf $v$, we associate a compact interval $I(v) \subseteq [0,1]$ representing its geometric support, whose diameter is given by the product of contraction factors along the path from the root to $v$.

Each Galton--Watson subtree inserted at a leaf has almost surely finite height, but the recursive replacement of leaves continues indefinitely. On the event of non-extinction, the resulting branching-process RIFS therefore has infinite depth and is not itself a Galton--Watson tree, but a recursively tree-valued random object.

Formally, for each leaf $v$ at depth $k-1$ with interval $I(v)$ and scale $s(v)=\mathrm{diam}(I(v))$, a random contraction factor $R_{k,v}$ is drawn from a distribution $\nu_{s(v)}$ supported on $(0,s(v))$. Independently, a finite Galton--Watson tree $\mathcal{G}_{k,v}$ is generated with i.i.d.\ offspring variables $\{N_u : u \in \mathcal{G}_{k,v}\}$ having distribution $\mathbb{P}(N=n)=p_n$. The offspring distribution has finite support, ensuring a bounded number of children per leaf, and each generated subtree has almost surely finite height.

At each replacement event at a leaf $v$, a single \emph{subtree-scale} contraction $R_{k,v}$ is drawn, fixing the diameter of the interval into which the entire finite Galton--Watson subtree is embedded. Conditional on this choice, the internal geometry of the subtree is generated by independent and identically distributed child-scale contractions and translations. Equivalently, each finite subtree is first realized as a similarity iterated function system on $[0,1]$ and then composed with the subtree-scale similarity map determined by $R_{k,v}$.

The entire subtree $\mathcal{G}_{k,v}$ is embedded inside a compact interval $I_{k,v} \subseteq I(v)$ of diameter $R_{k,v}$. Each child interval of $\mathcal{G}_{k,v}$ is embedded inside $I(k,v)$ with diameter determined by an independent contraction factor drawn from a distribution supported on $(0,s(v))$. The uniform-subtree case of \citep{HudnallDSouza2025} is a special instance of the present construction; an animation of the recursive system can be found in that work. Because $R_{k,v}<s(v)$ almost surely, every subtree is generated at a scale strictly smaller than that of its parent leaf, producing a strictly nested hierarchy.

Equivalently, the construction can be expressed in the canonical IFS formalism: each embedded subtree corresponds to a similarity map on its parent interval of the form
\[
w_{k,v}(x)=a_{k,v}+r_{k,v}x,\qquad 0<r_{k,v}<1,
\]
where $r_{k,v}=R_{k,v}/s(v)$.

All statements below are understood to hold almost surely on the event of non-extinction of the underlying branching process.

\subsection{Variants}

We consider two variants of the construction, which differ only in how the embedded subtrees are positioned within their parent intervals. In the \emph{non-anchored} case, the translation parameters $a_{k,v}$ are chosen at random in the admissible range $[0,s(v)-R_{k,v}]$, so that each subtree is embedded freely within its parent interval. This produces a nontrivial compact attractor contained in $[0,1]$. In the \emph{anchored} case, the translations are fixed at $a_{k,v}=0$, so that each embedded subtree includes the global boundary point $0$. In this setting, every infinite descending chain converges to the singleton $\{0\}$, although finite-depth approximations retain nontrivial structure.

Let $\mu_n$ denote the natural probability measure obtained at depth $n$, defined by assigning weights proportional to the products of contraction factors along each root-to-leaf path. In Variant~A (non-anchored), the sequence $\{\mu_n\}$ converges in distribution to a nontrivial limiting measure $\mu$. In Variant~B (anchored), tangent measures associated with finite-depth approximations converge in law to the same multifractal class as in Variant~A. This behavior is consistent with the standard multifractal formalism \citep{Mandelbrot1974}, here realized within the present branching-process RIFS construction.

\subsection{Relation to classical branching and random fractal models}

The branching-process RIFS introduced here is related to, but distinct from, several well-studied classes of random constructions.

It is not a standard Galton--Watson branching process, which describes genealogical branching without associated geometric structure. Nor is it a classical random iterated function system in the sense of \cite{Falconer1997}, \cite{MauldinWilliams1986}, or related random recursive constructions, where a fixed or randomly chosen family of contraction mappings acts on a static geometric support. In the branching-process RIFS, by contrast, the family of contractions at a given scale is itself generated by a branching process, and new contraction families are created recursively along each lineage.

Similarly, while branching random walks and multiplicative cascades assign random weights or displacements to the vertices of a branching tree embedded in a fixed ambient space, they do not recursively generate new geometric environments along lineages. In the present model, genealogical branching actively drives the creation of new geometric structure at every scale, yielding a dynamically evolving cascade of supports.

The novelty of the construction lies in this recursive coupling of genealogical branching and geometric embedding. Although the multifractal analysis ultimately relies on probabilistic techniques related to those developed for random cascades and RIFS, the underlying object falls outside the scope of existing models because its geometry is regenerated endogenously at each scale.

\subsection{Nestedness, duality, and randomness}\label{subsec:three-principles}

The construction gives precise mathematical form to three principles—nestedness, duality, and randomness—that are well established in the biological theory literature~\citep{Kimura1983,Lewontin1970,SzathmarySmith1995}. Nestedness arises because each contraction factor satisfies $R_{k,v}<s(v)$ almost surely, so every child interval is strictly contained in its parent. Along any root-to-leaf path $v_0,\ldots,v_n$, the associated intervals satisfy
\[
I(v_n)\subset I(v_{n-1})\subset\cdots\subset I(v_0)=[0,1],
\]
which expresses conditional dependence in exact mathematical terms.

Duality is built into the recursive replacement process: each object is a leaf with respect to its parent, yet simultaneously the root of a subtree with respect to its descendants. Its role depends on the level of resolution at which the system is observed, capturing the whole--part character of biological systems. Randomness enters through both the offspring numbers $\{N_u\}$ and contraction factors $\{R_{k,v}\}$, which are independent random variables across leaves. Their independence ensures stochasticity at every stage, making the overall process intrinsically probabilistic. Taken together, these features show that the branching-process RIFS provides a mathematical integration of nestedness (via strict contraction), duality (via branching), and randomness (via stochastic offspring and scale).

\subsection{Assumptions and biological restrictions}\label{subsec:assumptions}

While the construction formalizes well-established biological principles, several simplifying assumptions restrict its biological scope. Reproduction is strictly branching, so the model excludes fusion processes and describes only asexual reproduction. Each entity reproduces once, at the end of its lifetime, so parent and child lifespans do not overlap. The offspring distribution has finite support, ruling out heavy-tailed reproductive variation observed in some biological systems.

For the mathematical analysis, we impose mild non-degeneracy conditions. The offspring distribution is not concentrated on $0$ or $1$, ensuring positive probability of at least two offspring. Offspring numbers and contraction factors are taken to be independent both across leaves and with respect to one another. Contraction factors are continuous random variables supported on $(0,1)$ with non-degenerate distributions. These conditions guarantee that the attractor is uncountable and that the limiting measure is nontrivial almost surely. All claims in the body of the paper apply specifically to this model class.

\section{Multifractality of the branching-process RIFS}\label{sec:results}

We now state and prove the central result of this paper. Recall the construction in Section~\ref{sec:model}, with Galton--Watson offspring variables $\{N_u\}$, continuous contraction factors $\{R_{k,v}\}$ drawn independently from distributions supported on $(0,s(v))$, and recursive embedding of finite subtrees within their parent intervals.

\begin{theorem}[Multifractality of the branching-process RIFS]\label{thm:multifractality}
Consider the branching-process RIFS under the assumptions of Section~\ref{subsec:assumptions}. In Variant~A (non-anchored), the sequence of measures $\{\mu_n\}$ converges almost surely to a limiting random probability measure $\mu$ supported on a nontrivial compact attractor in $[0,1]$. The associated $L^q$-spectrum
\begin{equation}\label{eq:tau}
\tau(q) \;=\; \lim_{n \to \infty} \frac{\log Z_n(q)}{\log \varepsilon_n}
\end{equation}
exists almost surely, is strictly convex on an interval $(q_-,q_+)$, and its Legendre transform yields a nontrivial multifractal spectrum $f(\alpha)$. Hence $\mu$ is multifractal.

In Variant~B (anchored), the invariant geometric set reduces to the singleton $\{0\}$. Nevertheless, tangent measures obtained by affine rescaling of finite-depth neighborhoods converge in law to the same multifractal class as in Variant~A. In particular, the anchored case exhibits the same $L^q$-spectrum $\tau(q)$ and multifractal spectrum $f(\alpha)$ through its family of tangent measures.
\end{theorem}

We first give a brief proof sketch, followed by the full argument.

\paragraph{Proof sketch.}
The proof follows the multifractal cascade framework, adapted to the present genealogically generated geometry. The weights along root-to-leaf paths define a multiplicative cascade, whose martingale convergence yields a limiting measure $\mu$ in Variant~A. Depth-wise partition sums satisfy a subadditive structure, implying almost sure existence of the pressure function $\tau(q)$. Non-degeneracy of the contraction law yields strict convexity of $\tau$ on an interval, and the multifractal formalism then gives a nontrivial spectrum $f(\alpha)$. In Variant~B, affine rescaling at finite depth regenerates the same stochastic environment as in Variant~A, so the same multifractal law holds for tangent measures.

\subsection{Existence of the limiting measure and attractor structure}\label{subsec:limit-measure}

We begin by establishing the existence of the limiting measure and describing the structure of its support in the non-anchored case.

\begin{lemma}\label{lem:limitingmeasure}
In Variant~A (non-anchored), under the assumptions of Section~\ref{subsec:assumptions}, the sequence of measures $\{\mu_n\}$ converges almost surely (after normalization) to a random probability measure $\mu$ supported on a compact attractor $\mathcal{A}\subset[0,1]$. The set $\mathcal{A}$ is totally disconnected and uncountable.
\end{lemma}

\begin{proof}
The full proof is given in Appendix~\ref{app:A1}. \qedhere
\end{proof}

\subsection{Pressure and $L^q$-spectrum}\label{subsec:pressure}

We quantify scaling via depth-wise partition functions and show that the limiting growth rate (the $L^q$-spectrum) exists almost surely.

\begin{definition}[Partition function and mesh scale]\label{def:partition}
Let $\{I_i^{(n)}\}_{i=1}^{L_n}$ be the leaf intervals at depth $n$, and define the partition function
\begin{equation}\label{eq:partition}
Z_n(q) \;=\; \sum_{i=1}^{L_n} \mu\!\left(I_i^{(n)}\right)^{q}, 
\qquad q \in \mathbb{R}.
\end{equation}
Define the mesh scale at depth $n$ by
\begin{equation}\label{eq:mesh}
\varepsilon_n \;=\; \max_{1 \le i \le L_n} \mathrm{diam}\!\left(I_i^{(n)}\right).
\end{equation}
A summary of symbols is collected in Appendix~\ref{App:C}.
\end{definition}

\begin{lemma}[Existence of the $L^q$-spectrum]\label{lem:existence-spectrum}
Under the assumptions of Section~\ref{subsec:assumptions}, for every $q\in\mathbb{R}$ the limit
\begin{equation}\label{eq:tauspectrum}
\tau(q) \;=\; \lim_{n \to \infty} \frac{\log Z_n(q)}{\log \varepsilon_n}
\end{equation}
exists almost surely and is finite (possibly taking the value $-\infty$ outside its effective domain).
\end{lemma}

\begin{proof}
The full proof is given in Appendix~\ref{app:A2}.
\end{proof}

\noindent
Although the limiting measure $\mu$ is random, the $L^q$-spectrum $\tau(q)$ is almost surely deterministic, as a consequence of the subadditive ergodic arguments underlying Lemma~\ref{lem:existence-spectrum}.

\begin{lemma}[Domain and convexity of $\tau$]\label{lem:convexity}
Under the assumptions of Section~\ref{subsec:assumptions}, there exists an open interval $(q_-,q_+)\ni 0$ such that $\tau(q)$ is finite for all $q\in(q_-,q_+)$. On this interval, $\tau$ is differentiable and strictly convex, with derivative $\alpha(q)=\tau'(q)$.
\end{lemma}

\begin{proof}
The full proof is given in Appendix~\ref{app:A3}.
\end{proof}

\begin{definition}[Multifractal spectrum]\label{def:multispectrum}
Define the multifractal spectrum by the Legendre transform~\citep{Mandelbrot2003}:
\begin{equation}\label{eq:multispectrum}
f(\alpha) \;=\; \inf_{q \in \mathbb{R}} \left( q\alpha - \tau(q) \right),
\qquad
\alpha(q) = \tau'(q) \ \text{when $\tau$ is differentiable}.
\end{equation}
\end{definition}

\begin{corollary}[Nontrivial multifractal spectrum]\label{cor:nontrivial}
Under the assumptions of Lemma~\ref{lem:convexity}, the multifractal spectrum $f(\alpha)$ is nontrivial (positive on an interval of $\alpha$).
\end{corollary}

\begin{theorem}[Multifractal formalism on an interval]\label{thm:formalism}
Assume the conditions of Section~\ref{subsec:assumptions} and, in addition, a quasi-Bernoulli property for the leaf weights (equivalently, a weak Gibbs-type product structure of the cascade). Then the multifractal formalism holds on the interval $(q_-,q_+)$: for every $q\in(q_-,q_+)$, with $\alpha(q)=\tau'(q)$,
\[
f(\alpha(q)) \;=\; q\alpha(q) - \tau(q) \;=\; \tau^\ast(\alpha(q)),
\]
where $\tau^\ast$ denotes the Legendre transform of $\tau$.
\end{theorem}

\begin{proof}
The full proof is given in Appendix~\ref{app:A4}.
\end{proof}

\subsection{Anchored case: tangent measures and multifractality}

\paragraph{Interpretation via tangent measures.}
Although the invariant geometric set of the anchored construction collapses to $\{0\}$, the local structure observed through affine rescaling at finite depth regenerates the same stochastic environment as in the non-anchored case, as formalized below.

\begin{theorem}[Tangent-measure equivalence in the anchored case]\label{thm:tangent}
In Variant~B (anchored), for any sequence of depths $n\to\infty$ and any choice of a depth-$n$ leaf $v_n$, let
\[
T_{n,v_n} : I(v_n) \to [0,1]
\]
be an affine rescaling that maps the parent interval $I(v_n)$ to $[0,1]$. Define the rescaled measures
\begin{equation}\label{eq:tangentmeasure}
\mu_{v_n}^{(n)}(A) \;=\; \frac{\mu\!\left(T_{n,v_n}^{-1}(A) \cap I(v_n)\right)}{\mu(I(v_n))},
\qquad A \subset [0,1].
\end{equation}
Then any weak limit point of $\{\mu_{v_n}^{(n)}\}$ (along any subsequence) has the same distribution as the non-anchored limiting measure $\mu$ of Variant~A. In particular, the family of tangent measures in Variant~B is almost surely multifractal with the same $L^q$-spectrum $\tau(q)$ and multifractal spectrum $f(\alpha)$ as established in Lemmas~\ref{lem:existence-spectrum}--\ref{lem:convexity} and Theorem~\ref{thm:formalism} for Variant~A.
\end{theorem}

\begin{proof}
A full proof is given in Appendix~\ref{app:A5}.
\end{proof}

\begin{corollary}[Anchored multifractality via tangents]\label{cor:anchored}
Although the invariant geometric set is $\{0\}$ in Variant~B, the class of tangent measures at $\{0\}$ is almost surely multifractal with spectrum $f(\alpha)$ identical to Variant~A.
\end{corollary}

\subsection{Summary}

The results above complete the proof of Theorem~\ref{thm:multifractality}. In the non-anchored variant, Lemma~\ref{lem:limitingmeasure} establishes that $\{\mu_n\}$ converges almost surely to a limiting random probability measure $\mu$ supported on a compact, totally disconnected, uncountable attractor. Lemmas~\ref{lem:existence-spectrum} and~\ref{lem:convexity}, together with Theorem~\ref{thm:formalism}, show that the associated $L^q$-spectrum $\tau(q)$ exists almost surely, is strictly convex on $(q_-,q_+)$, and yields a nontrivial multifractal spectrum $f(\alpha)$ via Legendre transform. Theorem~\ref{thm:tangent} extends these conclusions to the anchored variant by showing that, although the invariant geometric set collapses to $\{0\}$, affine rescalings at finite depth regenerate the same multifractal law through tangent measures.

\subsection{Numerical illustration}\label{subsec:numerical-illustration}

Although the main results are established analytically, a simulation helps visualize the branching-process RIFS. Figure~\ref{fig:numerical} shows a typical realization to depth~20 (2,270 leaves at the final iterate), with contraction ratios and translations drawn independently from uniform distributions. The top panel shows the distribution of probability mass across normalized position and depth, and the bottom panel shows an estimated multifractal spectrum obtained from partition sums and the Legendre transform. Computational details are given in Appendix~\ref{app:B4}.

\begin{figure}[htbp]
  \centering
  \includegraphics[width=0.9\textwidth]{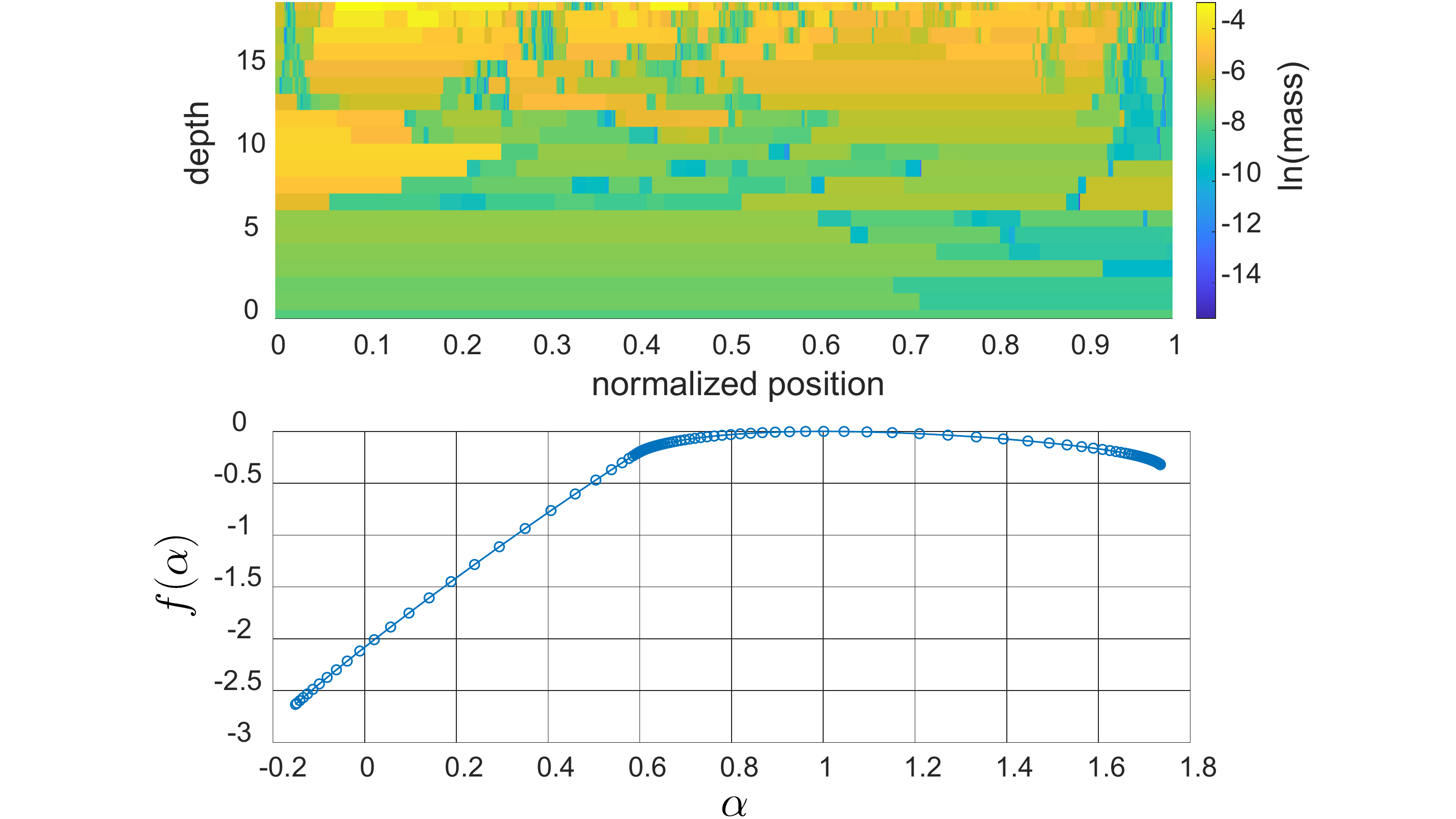}
  \caption{Numerical illustration of the branching-process RIFS. \textbf{Top:} Distribution of interval masses across normalized position and iteration depth (logarithmic scale; warmer colors indicate greater mass). Each row corresponds to a generation, showing the recursive embedding of lineages into smaller scales. \textbf{Bottom:} Estimated multifractal spectrum $f(\alpha)$, obtained from partition sums and Legendre transform of the scaling exponents. The strictly convex curve is consistent with Theorem~\ref{thm:formalism}, showing the coexistence of multiple scaling exponents within a single realization.}
  \label{fig:numerical}
\end{figure}

The numerical spectrum in Figure~\ref{fig:numerical} provides a concrete visualization of the nontrivial multifractality established analytically: even finite-depth realizations exhibit a broad range of effective scaling exponents within a single realization.

\section{Discussion}\label{sec:discussion}

These results connect the branching-process RIFS to several classical lines of work. On the mathematical side, they place it within the same general regime as multiplicative cascades while extending them to a genealogically dynamic Galton--Watson setting. Random iterated function systems have long been known to produce multifractal measures \citep{BarralJin2010,Falconer1997,Graf1987,KahanePeyriere1976,Mandelbrot1974}, whereas branching processes capture stochastic reproduction without generating multifractality on their own. The branching-process RIFS synthesizes the two: each leaf initiates a subtree at a smaller scale, producing a recursive cascade of genealogically nested trees. The resulting attractor is genealogically structured, with each point corresponding to the limit of a lineage. The anchored variant adds a further feature by collapsing the invariant set to $\{0\}$ while tangent measures retain multifractality. Taken together, these properties define a class of random fractal constructions distinct from classical Mandelbrot cascades.

Biologically, the model is motivated by the tree of life. Standard phylogenetic models are typically static realizations of genealogies, constructed via coalescent or birth--death processes~\citep{Hudson1990,Stadler2010}. By contrast, the present construction is dynamic: every leaf is simultaneously the root of a subtree, so that phylogeny is recursively generated rather than sampled at once. Between the two variants, the anchored case is the biologically motivated one: living systems are finite, and no lineage persists indefinitely. The collapse of the invariant set therefore encodes finiteness and mortality, while the persistence of multifractality through tangent measures shows that complex scaling structure arises at finite depth, precisely while the system is extant.

The proof establishes that multifractality is a consequence of nestedness, duality, and randomness, rather than an added hypothesis. As a minimal model, the branching-process RIFS provides a baseline theorem: more elaborate biological mechanisms can be expected to enrich, rather than suppress, the multifractal spectrum. In this sense, multifractality emerges as a structurally robust property of genealogically organized systems.

This perspective helps explain why multifractal signatures are repeatedly observed in biological data. In physiological time series such as human heartbeat intervals, scaling analyses reveal a spectrum of local Hölder exponents rather than a single fractal dimension~\citep{Ivanov1999}. Ecological succession produces concave multifractal spectra~\citep{Saravia2012}, and genomic sequence analyses show broad distributions of scaling exponents across the human genome~\citep{Moreno2011}. Reviews of fractal and multifractal methods document similar findings across biological imaging and physiological systems~\citep{LopesBetrouni2009}. These observations are consistent with the multifractal spectrum $f(\alpha)$ established in Theorem~\ref{thm:formalism}, suggesting that the empirical ubiquity of multifractality reflects the fundamental organizational principles of living systems.

Relaxing the model's simplifying assumptions—for instance, allowing heavy-tailed offspring distributions, heterogeneous contraction laws, or higher-dimensional embeddings—may further enrich the multifractal geometry and broaden its biological relevance. These directions point to a wider mathematical landscape in which multifractality is not only persistent but deepened by additional mechanisms of life.

\section{Conclusion}\label{sec:conclusion}

We have shown that a branching-process random iterated function system, constructed from Galton--Watson reproduction coupled with continuous geometric contractions, gives rise to multifractality under minimal and explicit assumptions. In the non-anchored variant, the limiting measure is supported on a nontrivial Cantor-like attractor and exhibits a strictly convex $L^q$-spectrum. In the biologically motivated anchored variant, although the invariant geometric set collapses to the singleton $\{0\}$, multifractality persists through tangent measures obtained from finite-depth rescalings.

From a mathematical perspective, the result extends classical multifractal cascade theory to a genealogically dynamic setting, in which branching structure and geometry are generated recursively rather than imposed on a fixed support. The branching-process RIFS defines a class of random fractal measures that lie outside standard Galton--Watson trees and classical random iterated function systems, while remaining amenable to analysis via established probabilistic techniques.

From a biological perspective, the theorem makes explicit the scaling consequences of principles that are already central to evolutionary theory. If living systems are organized through nested descent, whole–part duality, and stochastic reproduction~\citep{Kimura1983,Lewontin1970,SzathmarySmith1995}, then multifractal structure arises naturally from these foundations. In this sense, the result provides a theoretical explanation for the widespread appearance of multifractal scaling in biological data, situating it as a structural feature of genealogically organized systems rather than as an artifact of particular models or measurements.

\bmhead{Acknowledgements}
The author thanks Raissa D'Souza for her guidance and mentorship throughout the doctoral studies, of which this work forms a part. Gratitude is also extended to Guram Mikaberidze for many helpful conversations, and to the Biological and Agricultural Engineering Department at UC Davis for supporting this research and providing an enriching academic environment.

\section*{Declarations}
\textbf{Funding} This research received no external funding.\\
\textbf{Conflict of interest} The author declares no competing interests.\\
\textbf{Data availability} Data for the numerical example are publicly available at \url{https://github.com/KevinAndrewHudnall/Proof-of-ToL-Multifractality}.\\
\textbf{Code availability} Code for running the branching-process RIFS and calculating the spectra is available in the same repository as the data.\\
\textbf{Ethics approval/Consent/Materials/Author contribution} Not applicable.
\textbf{Use of generative AI in the writing process} During the preparation of this work, OpenAI's ChatGPT was used to study unfamiliar proof techniques, suggest relevant references, and provide feedback on structure and clarity after the author had constructed the full arguments. It was also used during final manuscript preparation to assist with text refinement, grammar checking, and rewording for clarity and concision. All conceptual content, mathematical arguments, and original writing were developed by the author. The AI was not used to generate new ideas, mathematical structures, figures, or to interpret data. The author takes full responsibility for the correctness and final content of the manuscript.

\begin{appendices}

\section{Additional proofs}\label{app:A}
\setcounter{equation}{0}
\setcounter{theorem}{0}

\numberwithin{equation}{section}
\renewcommand{\theequation}{A.\arabic{equation}}
\renewcommand{\thetheorem}{A.\arabic{theorem}}
\renewcommand{\thelemma}{A.\arabic{lemma}}
\renewcommand{\thecorollary}{A.\arabic{corollary}}
\renewcommand{\theremark}{A.\arabic{remark}}
This appendix provides the full proofs of Lemmas~\ref{lem:limitingmeasure}--\ref{lem:convexity} and Theorems~\ref{thm:formalism}--\ref{thm:tangent}.

\subsection{Proof of Lemma~\ref{lem:limitingmeasure} (Existence of the limiting measure)}\label{app:A1}
\begin{proof}
At each depth $n$, the measure $\mu_n$ assigns mass proportional to the product of contraction factors along each root-to-leaf path. Since all contraction factors lie in $(0,1)$ and satisfy $R_{k,v} < s(v)$, the diameters of intervals shrink strictly with depth. Thus each infinite path generates a nested sequence of closed intervals $\{I(v_n)\}$ with diameters tending to zero. By Cantor's intersection theorem~\citep{Cohn2013}, the intersection $\cap_{n=0}^\infty I(v_n)$ is a singleton.

The attractor $\mathcal{A}$ is defined as the set of all such limit points over all infinite paths. Because it is the intersection of a nested sequence of unions of closed intervals, $\mathcal{A}$ is closed. Being contained in $[0,1]$, it is also bounded, hence compact. By the non-degeneracy assumptions of Section~\ref{subsec:assumptions}---namely, that the offspring distribution assigns positive probability to at least two children and the contraction distribution assigns positive probability to at least two distinct values---distinct paths branch apart infinitely often, so $\mathcal{A}$ is totally disconnected and uncountable almost surely.

For convergence of measures, note that the weights $\mu_n(I(v))$ form a martingale with respect to the natural filtration generated by the branching process and contraction variables. By the Kahane--Peyrière martingale convergence theorem for multiplicative cascades~\citep{KahanePeyriere1976}, the martingale converges almost surely to a random probability measure $\mu$ supported on $\mathcal{A}$.
\end{proof}

\subsection{Proof of Lemma~\ref{lem:existence-spectrum} (Existence of the $L^q$-spectrum)}\label{app:A2}
\begin{proof}
Write $W_v$ for normalized contraction weights. Then $\mu(I_i^{(n)})$ is the product of the independent weights along the corresponding length-$n$ path, i.e., a multiplicative cascade.

Fix $q \in \mathbb{R}$. At depth $n$,
\begin{equation}
Z_n(q) \;=\; \sum_{\substack{\text{paths } |v| = n}} \; \prod_{j=1}^{n} W_{v_j}^{\,q}.
\end{equation}
Taking logs, $\log Z_n(q)$ is subadditive in $n$ (standard for branching random cascades), and its increments have finite expectation because the offspring law has finite support and $\log^+ W_v$ has finite moments by the support condition $R_{k,v}\in(0,s(v))$ (i.e., the bounded contraction assumption in Section~\ref{subsec:assumptions}). Hence Kingman's subadditive ergodic theorem~\citep{Kingman1973} yields the almost sure limit
\begin{equation}
\lim_{n\to\infty} \frac{1}{n}\,\log Z_n(q) \;=\; \kappa(q) 
\qquad \text{a.s.}
\end{equation}
for some finite function $\kappa(q)$. On the geometric side, $\varepsilon_n$ decays exponentially in $n$ almost surely because each refinement multiplies diameters by a random factor in $(0,1)$ bounded away from $1$ with positive probability; thus
\begin{equation}
\lim_{n\to\infty} \frac{1}{n}\,\log \varepsilon_n \;=\; \lambda \;<\;0
\qquad \text{a.s.}
\end{equation}
Combining,
\begin{equation}
\tau(q) \;=\; \lim_{n\to\infty} \frac{\log Z_n(q)}{\log \varepsilon_n}
\;=\;\frac{\kappa(q)}{\lambda}
\qquad \text{a.s.,}
\end{equation}
so the limit exists and is finite.
\end{proof}

\begin{remark}
Replacing $\varepsilon_n$ by any depth-wise representative scale with $\log \varepsilon_n \sim \lambda n$ (e.g., median or geometric mean diameter) leaves $\tau(q)$ unchanged. Thus $\tau$ is intrinsic to the cascade.
\end{remark}

\subsection{Proof of Lemma~\ref{lem:convexity} (Domain and convexity of $\tau$)}\label{app:A3}
We prove that there exists an open interval $(q_-,q_+)\ni 0$ on which $\tau(q)$ exists and is finite; that $\tau$ is convex on this interval and strictly convex when the contraction law is non-degenerate; and that $\tau$ is differentiable on the interior. The endpoints $q_\pm$ are determined by the moment behavior of the contraction law together with the bounded offspring distribution.

\subsubsection{Setup and a canonical normalization}
At each refinement, a parent node produces $N$ children (with $N$ supported on a finite set by Section~\ref{subsec:assumptions}) and i.i.d.\ contractions $R_1,\ldots,R_N \in (0,1)$. We use the standard Mandelbrot-cascade normalization
\begin{equation}\label{eq:canonW}
W_i \;=\; \frac{X_i}{\sum_{j=1}^{N} X_j}, 
\qquad
X_i \;=\; R_i^{\beta}, \ \ \beta>0,
\end{equation}
and define
\[
S(q) \;=\; \sum_{i=1}^{N} W_i^{\,q} \;=\; \frac{\sum_{i=1}^{N} X_i^{\,q}}{\bigl(\sum_{i=1}^{N} X_i\bigr)^q},
\qquad
Z_n(q) \;=\; \sum_{|v|=n} \prod_{j=1}^{n} W_{v_j}^{\,q}.
\]
Alternative positive, strictly increasing transforms of $R_i$ followed by renormalization change $\log Z_n(q)$ only by additive $O(n)$ terms with finite expectation (see Remark~\ref{rem:equiv-normalizations}), hence do not affect the existence/finiteness domain of $\tau$.

\subsubsection{Existence and identification of the exponential growth rate}
By Kingman's subadditive theorem~\citep{Kingman1973} plus the i.i.d.\ environment, the limit
\begin{equation}\label{eq:kappa-limit}
\kappa(q) \;=\; \lim_{n\to\infty} \frac{1}{n}\,\log Z_n(q)
\end{equation}
exists almost surely provided $\mathbb{E}\!\left(\log^{+} Z_1(q)\right)<\infty$. Moreover, the branching “many-to-one”/spine law of large numbers (e.g., \citep{BarralJin2010}) gives the identification
\[
\frac{1}{n}\log Z_n(q) \xrightarrow[\ n\to\infty\ ]{\ \mathrm{a.s.}\ } \mathbb{E}\!\left[\log S(q)\right],
\]
so that $\kappa(q)=\mathbb{E}\!\left[\log S(q)\right]$ whenever $\mathbb{E}\!\left[\,|\log S(q)|\,\right]<\infty$. The mesh scale decays exponentially, $\varepsilon_n=\exp(\lambda n+o(n))$ with $\lambda<0$ a.s.\ by Section~\ref{subsec:pressure}, hence
\begin{equation}\label{eq:tau-from-kappa}
\tau(q) \;=\; \lim_{n\to\infty} \frac{\log Z_n(q)}{\log \varepsilon_n} \;=\; \frac{\kappa(q)}{\lambda}
\end{equation}
exists and is finite iff $\mathbb{E}\!\left[\,|\log S(q)|\,\right]<\infty$. Thus the domain of $\tau$ is exactly $\{\,q:\ \mathbb{E}[\,|\log S(q)|\,]<\infty\,\}$.

\subsubsection{Nonnegative $q$: uniform bounds and $q_+=+\infty$}
Since $0<W_i\le 1$ and $\sum_i W_i=1$, for $q\ge 1$ convexity gives $S(q)=\sum_i W_i^{\,q}\le (\sum_i W_i)^q=1$, while the power-mean inequality gives $S(q)\ge N^{\,1-q}$. Hence
\[
(1-q)\log N_{\max} \;\le\; \log S(q) \;\le\; 0,
\]
with $N_{\max}<\infty$ the maximum support value of $N$. For $0\le q\le 1$, the concavity of $x\mapsto x^{q}$ gives $1\le S(q)\le N^{\,1-q}$, hence
\[
0 \;\le\; \log S(q) \;\le\; (1-q)\log N_{\max}.
\]
Therefore $\mathbb{E}[\,|\log S(q)|\,]<\infty$ for all $q\ge 0$; in particular $q_+=+\infty$ under the canonical normalization \eqref{eq:canonW}.

\subsubsection{Negative $q$: lower-tail critical exponent and $q_-$}
Let $q=-t<0$ with $t>0$. Using $X_i=R_i^{\beta}$,
\[
\log S(-t) \;=\; t \log\!\Bigl(\sum_{i=1}^{N} X_i\Bigr) \;+\; \log\!\Bigl(\sum_{i=1}^{N} X_i^{-t}\Bigr).
\]
The first term is harmless: $X_1\in(0,1)$ so $\log^{+}X_1=0$, and since $\sum_i X_i\le N_{\max}$ we have $\log^{+}(\sum_i X_i)\le \log N_{\max}$. Potential divergence comes only from $\sum_i X_i^{-t}$, i.e., from the lower tail of $X_i$ (equivalently, of $R_i$ near $0$). By standard moment theory (e.g., \citep{Billingsley2012}), there exists a (possibly zero) critical exponent $t_\star$. Set
\[
t_\star \;\equiv\; \sup\{\, t>0:\ \mathbb{E}[X_1^{-t}]<\infty \,\}
\;=\; \sup\{\, t>0:\ \mathbb{E}[R_1^{-\beta t}]<\infty \,\}\in[0,\infty].
\]
Using $\sum_i X_i^{-t}\le N_{\max}\,X_{(1)}^{-t}$ with $X_{(1)}=\max_i X_i$,
\[
\log^{+}\!\Bigl(\sum_{i} X_i^{-t}\Bigr) \;\le\; \log N_{\max} + \log^{+}\!\bigl(X_{(1)}^{-t}\bigr)
\;\le\; \log N_{\max} + X_{(1)}^{-t}.
\]
Thus $\mathbb{E}\!\left[\log^{+}(\sum_i X_i^{-t})\right]<\infty$ whenever $\mathbb{E}[X_1^{-t}]<\infty$. Conversely, if $\mathbb{E}[X_1^{-t}]=\infty$, then since $N\ge 1$ with positive probability and $X_1^{-t}\le \sum_i X_i^{-t}$, we get $\mathbb{E}\!\left[\log^{+}(\sum_i X_i^{-t})\right]=\infty$. Therefore
\[
\mathbb{E}\!\left[\,|\log S(-t)|\,\right]<\infty \ \Longleftrightarrow\ t<t_\star,
\quad\text{and}\quad
q_- \;=\; -t_\star.
\]

\subsubsection{Convexity, strict convexity, and differentiability}
For any fixed weights $W_1,\ldots,W_N\in(0,1)$ with $\sum_i W_i=1$, the map $q\mapsto \log\!\bigl(\sum_i W_i^{\,q}\bigr)$ is a log-sum-exp and therefore convex. Taking expectations preserves convexity, so $q\mapsto \kappa(q)=\mathbb{E}[\log S(q)]$ is convex on $(q_-,q_+)$; since $\lambda<0$, $\tau(q)=\kappa(q)/\lambda$ is likewise convex there.

If the contraction law is non-degenerate (supported on at least two distinct values with positive probability), then with positive probability the vector $W=(W_i)$ has at least two distinct coordinates, and the log-sum-exp is strictly convex on that event; averaging yields strict convexity of $\kappa$ (hence $\tau$) on $(q_-,q_+)$. (Equivalently, by Hölder; see, e.g., \citep{BarralJin2010,KahanePeyriere1976}, $\Lambda(q)=\log \mathbb{E}[R^q]$ is strictly convex when the $R$-law is non-degenerate, and this strictness transfers to $\kappa$ and $\tau$.)

Finally, differentiability of $\tau$ on the interior follows from standard convex analysis plus dominated differentiation for the random log-sum-exp: for $q$ in a compact subinterval of $(q_-,q_+)$,
\[
\frac{\partial}{\partial q}\log S(q)
\;=\;
\frac{\sum_{i} W_i^{\,q}\log W_i}{\sum_{i} W_i^{\,q}}
\]
has an integrable envelope (the negative tail is controlled by the same $t_\star$ margin used above), so $\kappa'(q)=\mathbb{E}[\partial_q \log S(q)]$ exists and is continuous, and $\tau'(q)=\kappa'(q)/\lambda$. This completes the proof.

\begin{remark}[Right endpoint]\label{rem:right-endpoint}
In the canonical normalization $W_i=X_i/\sum_{j} X_j$ with $X_i\in(0,1)$, we always have $\sum_{i} W_i^q\le 1$ for $q\ge 1$ and $\sum_{i} W_i^q\le N^{\,1-q}$ for $0\le q\le 1$. Hence $\log S(q)$ is bounded above and below uniformly for all $q\ge 0$. This guarantees not only finiteness but also convexity of $\tau$ on the entire half-line $[0,\infty)$. Thus $q_+=+\infty$. If a different but equivalent normalization is used (e.g., $W_i\propto R_i^\beta$ with a deterministic renormalization), the same conclusion holds because $\log S(q)$ changes by an additive term with finite expectation.
\end{remark}

\begin{remark}[Equivalence across normalizations]\label{rem:equiv-normalizations}
Suppose two weight schemes $W_i$ and $\widetilde{W}_i$ are defined by positive, continuous, strictly increasing transforms of each $R_i$, followed by normalization to sum to~$1$. Then, after \citep{KahanePeyriere1976}, there exist $c_1,c_2\in\mathbb{R}$ (depending on $q$) such that
\[
c_1 \;\le\; \log\!\Bigl(\sum_{i} W_i^{\,q}\Bigr) \;-\; \log\!\Bigl(\sum_{i} \widetilde{W}_i^{\,q}\Bigr) \;\le\; c_2
\quad \text{a.s.}
\]
Consequently, $\mathbb{E}[\,|\log S(q)|\,]<\infty$ (and the finiteness domain for $\kappa(q)$) is the same for both schemes; in particular, the endpoints $q_\pm$ are unchanged.
\end{remark}

\begin{remark}[Strict convexity condition]\label{rem:strict-convex}
The values of $q_\pm$ depend on the tails of the contraction distribution. If the distribution has finite support contained in $(0,1)$, then $\mathbb{E}[R^{-\beta t}]<\infty$ for all $t\ge 0$, so $q_-=-\infty$ and the domain $(q_-,q_+)$ extends across all of $\mathbb{R}$. In this case, $\tau$ remains convex everywhere but loses strict convexity: with only finitely many contraction values, $\tau$ can become affine and the multifractal spectrum collapses to a single point. More generally, strict convexity holds exactly when the contraction law is non-degenerate—that is, supported on at least two distinct values with positive probability. Non-degeneracy of the contraction distribution is therefore the precise condition for obtaining a nontrivial multifractal spectrum.
\end{remark}

\subsection{Proof of Theorem~\ref{thm:formalism} (Multifractal formalism on an interval)}\label{app:A4}
We prove that, under the assumptions of Section~\ref{subsec:assumptions} and a quasi\mbox{-}Bernoulli (weak Gibbs) property for leaf weights, the multifractal formalism holds on $(q_-,q_+)$: for every $q\in(q_-,q_+)$ with $\alpha(q)=\tau'(q)$,
\[
f(\alpha(q)) \;=\; q\,\alpha(q) - \tau(q) \;=\; \tau^\ast(\alpha(q)).
\]
Throughout, $\varepsilon_n$ denotes a representative mesh scale at depth $n$ with $\log \varepsilon_n \sim \lambda n$ and $\lambda<0$ a.s.\ (Remark~\ref{rem:right-endpoint}), and $I(v)$ a cylinder (leaf interval) at depth $|v|=n$.

\subsubsection{Roadmap}
By Lemma~\ref{lem:convexity}, $\tau$ is finite and convex on $(q_-,q_+)$, differentiable on the interior, and $\alpha(q)=\tau'(q)$ exists there. Thermodynamic-formalism arguments (e.g., \citep{BarralJin2010,KahanePeyriere1976}) yield a weak Gibbs property and identify the equilibrium (tilted) measures. The Legendre upper bound follows from a standard covering/counting step; the lower bound is obtained by constructing a Frostman measure supported on $E(\alpha(q))$ from the tilted measures. Together these give $f=\tau^\ast$ on $(q_-,q_+)$.

\subsubsection{Preliminaries: level sets and the Legendre upper bound}
Define the level sets of local scaling exponents
\[
E(\alpha)\;=\;\Bigl\{x:\lim_{r\downarrow 0}\frac{\log \mu(B(x,r))}{\log r}=\alpha\Bigr\}
\;=\;
\Bigl\{x:\ \lim_{n\to\infty}\frac{\log \mu(I(v_n(x)))}{\log \varepsilon_n}=\alpha\Bigr\},
\]
since $\log \varepsilon_n \sim \lambda n$ and diameters along cylinders are comparable to $\varepsilon_n$ (Section~\ref{subsec:limit-measure}).

For any $\alpha\in\mathbb{R}$ and $\eta>0$, split the depth-$n$ cylinders into
\[
\mathcal{C}_n(\alpha,\eta)\;=\;\Bigl\{I(v):\ \varepsilon_n^{\,\alpha+\eta}\ \le\ \mu(I(v))\ \le\ \varepsilon_n^{\,\alpha-\eta}\Bigr\}.
\]
For any $q\in\mathbb{R}$, set $Z_n(q)=\sum_{|v|=n}\mu(I(v))^q$. If $q>0$, then for $I(v)\in \mathcal{C}_n(\alpha,\eta)$ we have $\mu(I(v))\ge \varepsilon_n^{\alpha+\eta}$, hence
\[
Z_n(q)\ \ge\ \sum_{I(v)\in\mathcal{C}_n(\alpha,\eta)} \mu(I(v))^q
\ \ge\ \#\mathcal{C}_n(\alpha,\eta)\ \varepsilon_n^{\,q(\alpha+\eta)}.
\]
If $q<0$, then $\mu(I(v))\le \varepsilon_n^{\alpha-\eta}$ yields
\[
Z_n(q)\ \ge\ \sum_{I(v)\in\mathcal{C}_n(\alpha,\eta)} \mu(I(v))^q
\ \ge\ \#\mathcal{C}_n(\alpha,\eta)\ \varepsilon_n^{\,q(\alpha-\eta)}.
\]
Thus, in both cases,
\begin{equation}\label{eq:counting-bound}
\#\mathcal{C}_n(\alpha,\eta)\ \le\
\begin{cases}
Z_n(q)\ \varepsilon_n^{-\,q(\alpha+\eta)}, & q>0,\\[2pt]
Z_n(q)\ \varepsilon_n^{-\,q(\alpha-\eta)}, & q<0.
\end{cases}
\end{equation}
Taking $\limsup_{n\to\infty}$, using $\displaystyle \lim_{n\to\infty}\frac{\log Z_n(q)}{\log \varepsilon_n}=\tau(q)$, we obtain
\[
\limsup_{n\to\infty}\frac{\log \#\mathcal{C}_n(\alpha,\eta)}{\log \varepsilon_n}
\ \le\
\begin{cases}
-\,q\alpha\;-\;q\eta\;+\;\tau(q), & q>0,\\[2pt]
-\,q\alpha\;+\;q\eta\;+\;\tau(q), & q<0.
\end{cases}
\]
Letting $\eta\downarrow 0$ and optimizing over $q\in\mathbb{R}$ gives the standard Legendre upper bound
\begin{equation}\label{eq:legendre-upper}
\dim_H E(\alpha)\ \le\ \inf_{q\in\mathbb{R}} \{\, q\alpha - \tau(q)\,\}\ =\ \tau^\ast(\alpha).
\end{equation}
Thus it remains to prove the matching lower bound $\dim_H E(\alpha(q))\ge \tau^\ast(\alpha(q))$ for $q\in(q_-,q_+)$.

\subsubsection{Quasi-Bernoulli/Gibbs property and tilted measures}
Assume a quasi\mbox{-}Bernoulli (weak Gibbs) property: there exists $C\ge 1$ such that for all concatenated cylinders $I(uv)$,
\begin{equation}\label{eq:qB}
C^{-1}\,\mu(I(u))\,\mu_u(I(v))\;\le\;\mu(I(uv))\;\le\;C\,\mu(I(u))\,\mu_u(I(v)),
\end{equation}
where $\mu_u$ is the descendant law starting at node $u$ (same distribution as $\mu$ by independence). This is the standard quasi-product structure for random cascades and holds in the present i.i.d.\ setting (see \citep{BarralJin2010,KahanePeyriere1976}).

For fixed $q\in(q_-,q_+)$, define the depth-$n$ tilted (Gibbs) measure on cylinders by
\[
\mathbb{Q}_{q,n}(I(v)) \;=\; \frac{\mu(I(v))^q}{Z_n(q)}.
\]
Let $\mathbb{Q}_q$ be any subsequential weak limit (tightness is standard since the state space is compact and $Z_n(q)$ is finite on $(q_-,q_+)$). Under $\mathbb{Q}_q$, typical cylinders satisfy a Shannon--McMillan--Breiman/Kingman law:
\[
\lim_{n\to\infty} \frac{1}{n}\,\log \mu(I(v_n)) \;=\; \frac{\partial}{\partial q}\Bigl(\frac{1}{n}\log Z_n(q)\Bigr)
\qquad \mathbb{Q}_q\text{-a.s.},
\]
which, by Lemma~\ref{lem:convexity} and dominated convergence, equals $\kappa'(q)=\lambda\,\tau'(q)=\lambda\,\alpha(q)$. Therefore,
\[
\lim_{n\to\infty} \frac{\log \mu(I(v_n))}{\log \varepsilon_n} \;=\; \alpha(q)
\qquad \mathbb{Q}_q\text{-a.s.},
\]
so $\mathbb{Q}_q\!\bigl(E(\alpha(q))\bigr)=1$.

\subsubsection{Frostman measure and the lower bound}
To obtain a lower bound on $\dim_H E(\alpha(q))$, construct a Frostman measure supported on $E(\alpha(q))$. Define $\nu_q$ as any weak limit of $\mathbb{Q}_{q,n}$. Equivalently, for cylinders,
\[
\nu_q(I(v)) \;=\; \lim_{n\to\infty} \mathbb{Q}_{q,n}(I(v))
\ \asymp\ \frac{\mu(I(v))^q}{Z_{|v|}(q)}
\ \asymp\ \varepsilon_{|v|}^{\,q\alpha(q)-\tau(q)},
\]
with constants independent of $n$ (by \eqref{eq:qB}); by the previous step, $\nu_q(E(\alpha(q)))=1$.

Let $x\in E(\alpha(q))$ and choose $n$ so that $\varepsilon_{n+1}<r\le \varepsilon_n$. In one dimension a ball $B(x,r)$ meets at most a bounded number $M$ (a geometric constant) of depth-$n$ cylinders. Using quasi-Bernoulli to compare neighboring cylinders and the estimate
\[
\nu_q\bigl(I(v_n(x))\bigr)\ \asymp\ \frac{\mu(I(v_n(x)))^q}{Z_n(q)}\ \asymp\ \varepsilon_n^{\,q\alpha(q)-\tau(q)},
\]
we obtain
\[
\nu_q(B(x,r))
\ \le\ C \sum_{\substack{I \subset B(x,r)\\ |I|=n}} \nu_q(I)
\ \le\ C' M\, \nu_q\bigl(I(v_n(x))\bigr)
\ \lesssim\ \varepsilon_n^{\,q\alpha(q)-\tau(q)}
\ \asymp\ r^{\,q\alpha(q)-\tau(q)}.
\]
By Frostman's lemma, $\dim_H E(\alpha(q))\ge q\alpha(q)-\tau(q)=\tau^\ast(\alpha(q))$. Combined with the upper bound \eqref{eq:legendre-upper}, this gives
\[
\dim_H E(\alpha(q)) \;=\; \tau^\ast(\alpha(q))
\qquad \text{for all } q\in(q_-,q_+).
\]
This completes the proof of Theorem~\ref{thm:formalism}.

\begin{remark}[Quasi-Bernoulli property in the branching-process RIFS]\label{rem:qB-holds}
In the i.i.d.\ setting of Section~2.2, the quasi\mbox{-}Bernoulli (weak Gibbs) property required in Theorem~\ref{thm:formalism} is automatically satisfied. At each node the offspring number is sampled independently from a finite-support Galton--Watson distribution, and the contraction factors are drawn i.i.d.\ in $(0,1)$. Because the subtree environment below any node has the same distribution as at the root, and because interval weights factor multiplicatively along paths, the product-structure estimates of \citep{KahanePeyriere1976} apply. In particular, for concatenated cylinders $I(uv)$ one has
\[
C^{-1}\,\mu(I(u))\,\mu_u(I(v))\;\le\;\mu(I(uv))\;\le\;C\,\mu(I(u))\,\mu_u(I(v))
\]
for a constant $C$ independent of $u,v$. Thus the cascade law underlying the branching-process RIFS is quasi\mbox{-}Bernoulli, and the hypotheses of Theorem~\ref{thm:formalism} are satisfied.
\end{remark}

\subsection{Proof of Theorem~\ref{thm:tangent} (Tangent-measure equivalence in the anchored case)}\label{app:A5}
Let $\mu_n$ denote the depth-$n$ cascade measure in Variant~B, and fix a leaf $v$ at depth $n$. Write $I(v)\subset[0,1]$ for its cylinder interval and let
\[
S_v:I(v)\longrightarrow[0,1]
\]
be the unique affine rescaling with $S_v(I(v))=[0,1]$.

\subsubsection{Tangent (blow-up) candidates}
For each $k\ge 0$, define the rescaled descendant measure
\[
\nu_{v,k}\;:=\;\frac{\mu_{n+k}\!\restriction_{I(v)}}{\mu_n(I(v))}\circ S_v^{-1}.
\]
Equivalently, $\nu_{v,k}$ assigns to each depth-$k$ subcylinder of $I(v)$ its $\mu$-mass relative to $I(v)$, transported to $[0,1]$. Tangent measures at $v$ are weak limits of $\nu_{v,k_j}$ along subsequences $k_j\to\infty$.

\subsubsection{Regeneration at a leaf}
By construction, below any node $v$ the environment is independent of the past and i.i.d.\ with the root law: offspring numbers are Galton--Watson with finite support; contractions lie in $(0,1)$ and are i.i.d.; siblings are independent. The anchoring rule in Variant~B fixes the absolute placement of children inside $I(v)$, but after composing with $S_v$ this placement is erased. Consequently, the law of the subtree under $v$ (viewed through $S_v$) coincides with the root law of Variant~A. Hence, for every fixed $k\ge 0$,
\[
\nu_{v,k}\ \stackrel{d}{=}\ \mu^{A}_k,
\]
where $\mu^A_k$ denotes the depth-$k$ cascade approximant in Variant~A (non-anchored).

\subsubsection{Tightness and identification of limits}
The sequence $(\mu^A_k)_{k\ge 0}$ is tight and converges weakly (by the standard multiplicative-cascade martingale convergence theorem, e.g.\ \citep{KahanePeyriere1976}). Thus $\mu^A_k\Rightarrow\mu^A$ for some random limit measure $\mu^A$. Since $\nu_{v,k}\stackrel{d}{=}\mu^A_k$ for each $k$, we have
\[
\nu_{v,k}\ \Rightarrow\ \nu_v \quad\text{with}\quad \nu_v\ \stackrel{d}{=}\ \mu^A.
\]
Therefore every tangent measure at $v$ has the same distribution as the non-anchored limit $\mu^A$.

\subsubsection{Conclusion}
Because tangent measures in Variant~B coincide in law with $\mu^A$, they inherit the same strictly convex $L^q$-spectrum $\tau(q)$ and multifractal spectrum $f(\alpha)$ established for Variant~A. This completes the proof of Theorem~\ref{thm:tangent}.

\begin{remark}[Direct tangent construction from the limit measure]\label{rem:direct-tangent}
One may equivalently pass directly from the Variant~B limit $\mu$ instead of $(\mu_n)$. For a point $x\in I(v)$, define
\[
\nu_{v,k}\;:=\;\mu\!\restriction_{I(v)}\circ S_v^{-1}
\]
after $k$ further refinements below $v$. This construction yields the same law as $\mu^A_k$, and hence $\nu_{v,k}\Rightarrow \mu^A$. The identification of tangent measures therefore holds equally at the level of the limit measure.
\end{remark}

\section{Corollaries and Examples}\label{app:B}
\setcounter{equation}{0}
\setcounter{theorem}{0}

\numberwithin{equation}{section}
\renewcommand{\theequation}{B.\arabic{equation}}
\renewcommand{\thetheorem}{B.\arabic{theorem}}
\renewcommand{\thelemma}{B.\arabic{lemma}}
\renewcommand{\thecorollary}{B.\arabic{corollary}}
\renewcommand{\theremark}{B.\arabic{remark}}
This appendix collects several direct consequences of the main results, along with an explicit worked example and the computational details behind the numerical illustration in Section~\ref{subsec:numerical-illustration}. 
The aim is to show how the abstract theorems of Section~\ref{sec:results} translate into concrete formulas, degeneracy cases, and tractable benchmarks for simulation.

\subsection{Dimensional formula}\label{app:B1}
For any $q \in (q_-,q_+)$, write $\alpha(q)=\tau'(q)$. 
The Hausdorff dimension of the level set
\[
E(\alpha)=\Bigl\{x:\ \lim_{n\to\infty}\frac{\log\mu\bigl(I^{(n)}(x)\bigr)}{\log\varepsilon_n}=\alpha\Bigr\}
\]
is almost surely given by
\[
\dim_H E(\alpha(q)) \;=\; f(\alpha(q)) \;=\; q\alpha(q)-\tau(q) \;=\; \tau^\ast(\alpha(q)).
\]
This is the standard multifractal formalism (see Section~\ref{subsec:pressure} Theorem~\ref{thm:formalism}) realized in the branching-process RIFS.

\subsection{Degeneracy and stability}\label{app:B2}
The spectrum collapses precisely when the induced sibling weights are almost surely constant. 
In particular, if the contraction law is degenerate (e.g.\ $R\equiv r\in(0,1)$), then $X_i=r^\beta$ for all siblings and $W_i\equiv1/N$, so $\tau$ is affine and $f$ is trivial (a point). 
Conversely, whenever the contraction law is non-degenerate (supported on at least two values with positive probability), the induced weights are non-trivial with positive probability and $\tau$ is strictly convex (cf. Lemma~\ref{lem:convexity}), yielding a nontrivial (strictly concave) spectrum. 
Small perturbations of the underlying laws preserve strict convexity, and the spectrum varies continuously (stability under weak/parametric perturbations with the moment bounds in Section~\ref{subsec:assumptions}).

\subsection{Worked example: explicit cascade}\label{app:B3}
Take deterministic binary splitting $N\equiv2$, and i.i.d.\ contractions
\[
\mathbb{P}(R=1/3)=\mathbb{P}(R=2/3)=1/2,
\]
with canonical normalization $W_i=X_i/\sum_{j} X_j$, $X_i=R_i^\beta$ (set $\beta=1$ for simplicity). 
Then at one step
\[
S(q)=\frac{R_1^q+R_2^q}{(R_1+R_2)^q}.
\]
Enumerating the four pairs $(R_1,R_2)\in\{1/3,2/3\}^2$ gives a closed form for the pressure
\[
\kappa(q)=\mathbb{E}\bigl[\log S(q)\bigr]
=\tfrac{1}{2}\,\log\!\left(2\cdot 2^{-q}\,\Bigl(\bigl(\tfrac{1}{3}\bigr)^q+\bigl(\tfrac{2}{3}\bigr)^q\Bigr)\right).
\]
(Indeed, when both are $1/3$ or both $2/3$, $S(q)=2\cdot 2^{-q}$; in the mixed cases, $S(q)=(1/3)^q+(2/3)^q$.) 

With the exponential mesh rate $\varepsilon_n=\exp(\lambda n+o(n))$ from Section~\ref{subsec:pressure}, we have (compare to Definition~\ref{def:multispectrum}):
\[
\tau(q)=\frac{\kappa(q)}{\lambda}, 
\quad \alpha(q)=\tau'(q)=\frac{\kappa'(q)}{\lambda}, 
\quad f(\alpha(q))=q\,\alpha(q)-\tau(q).
\]
The resulting $f$ is a smooth, strictly concave curve on $[\alpha_{\min},\alpha_{\max}]$, providing a clean analytical benchmark for numerical estimates (see Figure~\ref{fig:numerical}).

\emph{Note.} If one replaces the canonical normalization by $W_i\propto R_i$ with a deterministic renormalization per level (non-random denominator), $\kappa$ shifts by a constant but the shape of $f$ (strict concavity, support) is unchanged.

\subsection{Computational details for the numerical illustration}\label{app:B4}
We constructed the scale matrix $S$ of interval diameters by depth. 
For each depth $n$ with at least two finite entries in $S(n,:)$, we computed
\begin{equation}
Z_n(q)=\sum_{j} S(n,j)^q.
\end{equation}
The representative scale was
\[
\varepsilon_n \;=\;\exp\Bigl(\tfrac{1}{\#\{j\}} \sum_{j}\log S(n,j)\Bigr),
\]
the geometric mean of $S(n,:)$.
The slope of the linear fit of $\log Z_n(q)$ vs.\ $\log\varepsilon_n$ across retained depths gives $\tau(q)$. 
The Legendre transform was obtained by finite differences,
\[
\alpha=\frac{d\tau}{dq},\qquad f(\alpha)=q\alpha-\tau.
\]
All computations were performed directly from the raw $S$, without renormalization.
\clearpage
\section{Notation table}\label{App:C}
\begin{table}[h!]
\centering
\caption{Summary of notation used throughout the paper.}
\begin{tabular}{lp{0.7\textwidth}}
\toprule
\textbf{Symbol} & \textbf{Meaning} \\
\midrule
\multicolumn{2}{l}{\textbf{Tree / Geometry}} \\
\addlinespace[2pt]
$X=[0,1]$ & \raggedright Ambient unit interval. \\
$I(v)$ & \raggedright Interval (cylinder) associated with node/leaf $v$. \\
$s(v)=\diam(I(v))$ & \raggedright Scale (diameter) of $I(v)$. \\
$L_n$ & \raggedright Number of depth-$n$ leaves (cylinders). \\
$\mathcal{G}_{k,v}$ & \raggedright Galton--Watson subtree rooted at $v$ at step $k$. \\
$N_u$ & \raggedright Offspring number at node $u$ (finite support). \\
\addlinespace[4pt]
\multicolumn{2}{l}{\textbf{Contractions / Similarity maps}} \\
\addlinespace[2pt]
$R_{k,v}$ & \raggedright Contraction factor for subtree at $(k,v)$, with law $\nu_{s(v)}$ on $(0,s(v))$. \\
$r_{k,v}=R_{k,v}/s(v)$ & \raggedright Normalized contraction ratio in $w_{k,v}$. \\
$a_{k,v}$ & \raggedright Translation parameter; $a_{k,v}\in[0,s(v)-R_{k,v}]$ (non-anchored), $a_{k,v}=0$ (anchored). \\
$w_{k,v}(x)=a_{k,v}+r_{k,v}x$ & \raggedright Similarity embedding of a subtree in its parent interval. \\
\addlinespace[4pt]
\multicolumn{2}{l}{\textbf{Measures}} \\
\addlinespace[2pt]
$\mu_n$ & \raggedright Depth-$n$ cascade probability measure. \\
$\mu$ & \raggedright Limiting measure in Variant~A (non-anchored). \\
$T_{n,v}$ & \raggedright Affine rescaling $I(v)\to[0,1]$ for tangents. \\
$\mu_{v}^{(n)}$ & \raggedright Rescaled/tangent measure around $v$ at depth $n$. \\
\addlinespace[4pt]
\multicolumn{2}{l}{\textbf{Partition sums / Mesh scale}} \\
\addlinespace[2pt]
$Z_n(q)$ & \raggedright Partition function $\displaystyle Z_n(q)=\sum_{i=1}^{L_n}\mu\!\bigl(I_i^{(n)}\bigr)^q$. \\
$\varepsilon_n$ & \raggedright Mesh scale at depth $n$ (e.g., max or geometric mean diameter). \\
$\kappa(q)$ & \raggedright Exponential growth rate $\displaystyle \kappa(q)=\lim_{n\to\infty}\tfrac{1}{n}\log Z_n(q)$. \\
$\lambda$ & \raggedright Mesh decay rate $\displaystyle \lambda=\lim_{n\to\infty}\tfrac{1}{n}\log \varepsilon_n<0$. \\
\addlinespace[4pt]
\multicolumn{2}{l}{\textbf{Spectral / Multifractal quantities}} \\
\addlinespace[2pt]
$\tau(q)$ & \raggedright $L^q$-spectrum $\displaystyle \tau(q)=\lim_{n\to\infty}\frac{\log Z_n(q)}{\log \varepsilon_n}=\kappa(q)/\lambda$. \\
$\alpha(q)=\tau'(q)$ & \raggedright Typical local scaling exponent associated with $q$. \\
$f(\alpha)$ & \raggedright Multifractal spectrum $f(\alpha)=\inf_{q\in\mathbb{R}}\{q\alpha-\tau(q)\}=\tau^\ast(\alpha)$. \\
$E(\alpha)$ & \raggedright Level set $\displaystyle E(\alpha)=\{x:\lim_{r\downarrow0}\tfrac{\log\mu(B(x,r))}{\log r}=\alpha\}$. \\
\addlinespace[4pt]
\multicolumn{2}{l}{\textbf{Cascade weights (one-step objects)}} \\
\addlinespace[2pt]
$X_i$ & \raggedright Positive transform of contractions (e.g., $X_i=R_i^{\beta}$). \\
$W_i=\dfrac{X_i}{\sum_j X_j}$ & \raggedright Canonical normalized sibling weights (Mandelbrot cascade). \\
$S(q)=\sum_i W_i^{\,q}$ & \raggedright One-step log-sum-exp generating function. \\
\addlinespace[4pt]
\multicolumn{2}{l}{\textbf{Domains / Parameters}} \\
\addlinespace[2pt]
$(q_-,q_+)$ & \raggedright Maximal open interval where $\tau(q)$ is finite (and differentiable in the interior). \\
$\tau^\ast$ & \raggedright Legendre transform of $\tau$: $\tau^\ast(\alpha)=\inf_q (q\alpha-\tau(q))$. \\
Variant~A / B & \raggedright Non-anchored (free translation) / Anchored ($a_{k,v}=0$) construction. \\
\bottomrule
\end{tabular}
\end{table}

\end{appendices}



\end{document}